\documentclass[a4paper,11pt,twoside]{amsart}
\usepackage[T1]{fontenc}
\usepackage[utf8]{inputenc}
\usepackage[british]{babel}
\usepackage{amsmath,amssymb,amsthm,amscd}
\usepackage{enumerate}
\usepackage{verbatim}
\usepackage{lmodern}
\usepackage{color}

\theoremstyle{definition}

\newtheorem{thm}{Theorem}[section]
\newtheorem{dfn}[thm]{Definition}
\newtheorem{lem}[thm]{Lemma}
\newtheorem{prp}[thm]{Proposition}
\newtheorem{cor}[thm]{Corollary}

\newtheorem{rem}[thm]{Remark}

\author{Tim de Laat}
\thanks{TdL is supported by the Deutsche Forschungsgemeinschaft -- Project-ID 427320536 -- SFB 1442, as well as under Germany’s Excellence Strategy -- EXC 2044 -- 390685587, Mathematics Münster: Dynamics -- Geometry -- Structure.}
\address{Tim de Laat, 
\newline Westf\"alische Wilhelms-Universit\"at M\"unster, Mathematisches Institut
\newline Einsteinstra\ss{}e 62, 48149 M\"unster, Germany}
\email{tim.delaat@uni-muenster.de}

\author{Safoura Zadeh}
\thanks{SZ is supported through the ``Oberwolfach Leibniz Fellows'' program -- Project-ID 2105q, by Mathematisches Forschungsinstitut Oberwolfach in 2021.}
\address{Safoura Zadeh,
\newline Max-Planck-Institut f\"{u}r Mathematik, Vivatsgasse 7, 53111, Bonn, Germany}
\email{jsafoora@gmail.com}

\date{}

\title[Weak$^*$-continuity of invariant means]{Weak$^*$-continuity of invariant means on spaces of matrix coefficients}

\begin{document}

\begin{abstract}
With every locally compact group $G$, one can associate several interesting bi-invariant subspaces $X(G)$ of the weakly almost periodic functions $\mathrm{WAP}(G)$ on $G$, each of which captures parts of the representation theory of $G$. Under certain natural assumptions, such a space $X(G)$ carries a unique invariant mean and has a natural predual, and we view the weak$^*$-continuity of this mean as a rigidity property of $G$. Important examples of such spaces $X(G)$, which we study explicitly, are the algebra $M_{\mathrm{cb}}A_p(G)$ of $p$-completely bounded multipliers of the Fig\`a-Talamanca--Herz algebra $A_p(G)$ and the $p$-Fourier--Stieltjes algebra $B_p(G)$. In the setting of connected Lie groups $G$, we relate the weak$^*$-continuity of the mean on these spaces to structural properties of $G$. Our results generalise results of Bekka, Kaniuth, Lau and Schlichting.
\end{abstract}

\maketitle

\section{Introduction} \label{sec:introduction}
With every locally compact group $G$, which in this article we assume to be second countable and Hausdorff, one can associate several interesting function spaces consisting of matrix coefficients of different classes of representations. In the setting of unitary representations, notable examples, which were introduced by Eymard \cite{MR0228628}, are the Fourier--Stieltjes algebra $B(G)$, consisting of all matrix coefficients of unitary representations of $G$, and the Fourier algebra $A(G)$, consisting of all matrix coefficients of the left-regular representation of $G$.

A much more general class of representations is the class of uniformly bounded representations of $G$ on reflexive Banach spaces. Albeit being a very large class, these representations still have nice analytic properties. For instance, if $\pi \colon G \to \mathcal{B}(E)$ is such a representation, then $E$ decomposes as the direct sum of the closed subspace of $\pi(G)$-invariant vectors and a canonical invariant complement \cite{MR0002026} (see also \cite{MR3191647}, \cite{MR3579953}). The algebra of matrix coefficients of such representations coincides with the algebra $\mathrm{WAP}(G)$ of weakly almost periodic functions on $G$ \cite{MR2077804} (see also \cite{MR0628827}, \cite{MR3692904}).

As a well-known consequence of the Ryll-Nardzewski fixed point theorem, the space $\mathrm{WAP}(G)$ carries a unique (two-sided) invariant mean (see e.g.~\cite[\S 3.1]{MR0251549}). In \cite[Theorem 3.6]{MR3552017}, Haagerup, Knudby and the first-named author showed that under mild conditions, a bi-invariant linear subspace $X(G)$ of $\mathrm{WAP}(G)$ has a unique invariant mean, arising as the restriction of the mean on $\mathrm{WAP}(G)$. A particular case of this result, namely the case $B(G)$, was already known from \cite{MR0023243}. We recall some background on weakly almost periodic functions and invariant means in Section \ref{sec:preliminaries}.

The Fourier--Stieltjes algebra $B(G)$ can be identified with the dual space of the universal group $C^*$-algebra $C^*(G)$ of $G$. It is known that the mean on $B(G)$ is weak$^*$-continuous if and only if $G$ has Kazhdan's property (T) \cite{MR0634144}, \cite{MR0773186} (see also \cite{MR3552017}), which is a well-known rigidity property for groups with many applications (see \cite{MR2415834}).

This characterisation of property (T) raises the question under which conditions a space $X(G)$ as above has a predual. We establish sufficient conditions for this in Theorem \ref{thm:predual} and describe an explicit predual $Y(G)$ for these spaces. Under the conditions that a (sufficiently large) function space $X(G)$ has a (unique) invariant mean $m_X$ and a natural predual $Y(G)$, we can study the analogue of property (T) corresponding to the weak$^*$-continuity of $m_X$. More precisely, we say that $G$ has property (T$^*_X$) if the invariant mean $m_X$ on $X(G)$ is weak$^*$-continuous with respect to the weak$^*$-topology coming from the predual $Y(G)$ (see also Definition \ref{dfn:tstarx}). This idea generalises both property (T), corresponding to $X(G)=B(G)$, and property (T$^*$) from \cite{MR3552017}, corresponding to $X(G)$ being the space $M_{\mathrm{cb}}A(G)$ of completely bounded Fourier multipliers.

We study property (T$^*_X$) explicitly in the concrete cases of $X(G)$ being the algebra $M_{\mathrm{cb}}A_p(G)$ of $p$-completely bounded multipliers of the Fig\`a-Talamanca--Herz algebra $A_p(G)$, and $X(G)$ being the $p$-Fourier--Stieltjes algebra $B_p(G)$, with $1 < p < \infty$. We recall these algebras in Section \ref{sec:preliminaries}. For these $X(G)$ and connected Lie groups $G$, we relate property (T$^*_X$) to the structure of $G$; see Theorem \ref{thm:mcbapstructureliegroups} and Theorem \ref{thm:bpstructureliegroups}. In these theorems, the space $X(G) \cap C_0(G)$ is of special interest. Indeed, as can be expected, under certain natural assumptions on $G$, property (T$^*_X$) is equivalent to the fact that $X(G) \cap C_0(G)$ is weak$^*$-closed in $X(G)$. This follows directly from the weak$^*$-continuity of the mean in combination with a very general Howe--Moore type theorem due to Veech \cite{MR0550072}. Note that the space $X(G) \cap C_0(G)$ can be viewed as a generalisation of the Rajchman algebra $B(G) \cap C_0(G)$.

Theorem \ref{thm:bpstructureliegroups} generalises a result of Bekka, Kaniuth, Lau and Schlichting \cite[Theorem 2.7]{MR1401762}, where for connected Lie groups $G$, the weak$^*$-closedness of the space $B(G) \cap C_0(G)$ in $B(G)$ was related to the structure of $G$. Let us point out that the characterisation of property (T) in terms of the weak$^*$-continuity of the mean on $B(G)$ was not used there.

Our initial aim was to prove results along the lines of Theorem \ref{thm:mcbapstructureliegroups} and Theorem \ref{thm:bpstructureliegroups} for more abstract classes of function spaces/algebras (of matrix coefficients) that carry an invariant mean and have a natural predual. However, in order to relate the weak$^*$-continuity of the mean to the structure of groups, we would have needed to put too many additional assumptions on $X(G)$, which would not have justified the level of abstraction.

\section*{Acknowledgements}
The work leading to this paper started whilst the second-named author was visiting the Westf\"alische Wilhelms-Universit\"at M\"unster. She would like to gratefully thank the WWU M\"unster for the hospitality and excellent working conditions during the visit.

\section{Preliminaries} \label{sec:preliminaries}

\subsection{Invariant means}
Let $G$ be a locally compact group and $X(G)$ a linear subspace of $L^{\infty}(G)$ that is closed under complex conjugation and that contains all constant functions. A positive linear functional $m \colon X(G) \rightarrow \mathbb{C}$ satisfying $\|m\|=m(1)=1$ is said to be a mean on $X(G)$.

We assume that $X(G)$ is invariant under left and right translations, i.e.~if $\varphi \in X(G)$, then for all $g \in G$, the functions $L_g\varphi \colon h \mapsto \varphi(g^{-1}h)$ and $R_g\varphi \colon h \mapsto \varphi(hg)$ lie in $X(G)$. A mean $m$ is called left invariant (resp.~right invariant) if $m(L_g\varphi)=m(\varphi)$ (resp.~$m(R_g\varphi)=m(\varphi)$) for all $\varphi \in X(G)$ and $g \in G$. A mean that is both left and right invariant is called two-sided invariant. In this article, invariant means are always assumed to be two-sided invariant, unless explicitly mentioned otherwise.

\subsection{Weakly almost periodic functions}
For a locally compact group $G$, a function $\varphi \in C_b(G)$ is called weakly almost periodic if the left orbit $O(\varphi,L)=\{L_g\varphi \mid g \in G\}$ (or equivalently, the right orbit $O(\varphi,R)=\{R_g\varphi \mid g \in G\}$) is relatively weakly compact, i.e.~its closure is compact in the weak topology on $C_b(G)$. We denote the space of weakly almost periodic functions by $\mathrm{WAP}(G)$. This space is a closed bi-invariant and inversion invariant subalgebra of $C_b(G)$ that contains the constants.

It is well known (see e.g.~\cite[\S 3.1]{MR0251549}) that for every locally compact group $G$, there exists a unique left invariant mean $m$ on $\mathrm{WAP}(G)$. Given an element $\varphi \in \mathrm{WAP}(G)$, its mean $m(\varphi)$ is explicitly given by the unique constant in the weakly closed convex hull $C(\varphi,L)$ of $O(\varphi,L)$ (which equals the unique constant in the analogously defined set $C(\varphi,R)$). Moreover, the mean $m$ is right invariant and inversion invariant. For details on weakly almost periodic functions, see \cite{MR0251549}, \cite{MR0263963}.

\subsection{An invariant mean on subspaces of $\mathrm{WAP}(G)$}
It is known that certain classes of subspaces of $\mathrm{WAP}(G)$ also carry a unique invariant mean. The following result is \cite[Theorem 3.6]{MR3552017}.
\begin{thm} \label{thm:restrictionofmeanonwap}
Let $G$ be a locally compact group and $X(G)$ a linear subspace of $\mathrm{WAP}(G)$ that is closed under left translations and conjugation and that contains the constants. Then $X(G)$ carries a unique left invariant mean $m_X$, which is in fact the restriction of the mean on $\mathrm{WAP}(G)$. If additionally, $X(G)$ is closed under right translations, then $m_X$ is right invariant as well. Moreover, if $X(G)$ is closed under inversion, then $m_X$ is also invariant under inversion.
\end{thm}

\subsection{The space $M_{\mathrm{cb}}A_p(G)$}
For $1 < p < \infty$, we denote by $M_{\mathrm{cb}}A_p(G)$ the space of $p$-completely bounded multipliers of the Fig\`a-Talamanca--Herz algebra $A_p(G)$ of $G$. We mainly work with the following characterisation of $p$-completely bounded multipliers from \cite{MR2606882}. Let $QSL^p$ denote the class of quotients of closed subspaces (or, equivalently, closed subspaces of quotients) of $L^p$-spaces.
\begin{prp}[{\cite[Theorem 8.3]{MR2606882}}] \label{prp:charpcbfm}
    Let $G$ be a locally compact group and $1 < p < \infty$. A function $\varphi \colon G \to \mathbb{C}$ is a $p$-completely bounded multiplier of $A_p(G)$ if there exists a space $E \in QSL^p$ and bounded, continuous maps $\alpha \colon G \to E$ and $\beta \colon G \to E^*$ such that for all $g,h \in G$,
\begin{equation} \label{eq:mcbapcharacterisation}
    \varphi(hg^{-1}) = \langle \beta(h),\alpha(g) \rangle.
\end{equation}
\end{prp}
Given $\varphi \in M_{\mathrm{cb}}A_p(G)$, its norm $\|\varphi\|_{M_{\mathrm{cb}}A_p(G)}$ is given by the infimum of the numbers $\|\alpha\|_{\infty}\|\beta\|_{\infty}$ over all choices of $E$, $\alpha$ and $\beta$ for which \eqref{eq:mcbapcharacterisation} holds. The space $M_{\mathrm{cb}}A_2(G)$ corresponds with the usual completely bounded multipliers of the Fourier algebra $A(G)$.

It is known that for $1 < p < \infty$, the space $M_{\mathrm{cb}}A_p(G)$ is a subalgebra of $\mathrm{WAP}(G)$ \cite{MR1373647} (see also \cite[Proposition 3.3]{MR3552017} for an explicit proof in the case $p=2$). Also, for $1 < p \leq q \leq 2$ or $2 \leq q \leq p < \infty$, the algebra $M_{\mathrm{cb}}A_q(G)$ embeds contractively into $M_{\mathrm{cb}}A_p(G)$ \cite[Proposition 6.1]{MR2652178}.

In general, a Fourier multiplier that is not completely bounded is not necessarily a weakly almost periodic function (see \cite{Bozejko}), which is why we cannot consider the space $MA_p(G)$ of such multipliers in the setting of this article.

\subsection{The space $B_p(G)$}
For a locally compact group $G$ and $1 < p < \infty$, let $\mathrm{Rep}_p(G)$ denote the collection of all (isometric equivalence classes of) isometric representations of $G$ on a $QSL^p$-space. Examples of elements in $\mathrm{Rep}_p(G)$ are the trivial representation (on every $QSL^p$-space) and the left-regular representation $\lambda_p \colon G \to \mathcal{B}(L^p(G))$.

Let $1 < p < \infty$. The $p$-Fourier--Stieltjes algebra $B_p(G)$ is defined as the space of matrix coefficients of (isometric equivalence classes of) isometric representations on a $QSL^p$-space, i.e.~functions of the form
\begin{equation} \label{eq:bpcharacterisation}
    g \mapsto \langle \eta, \pi(g)\xi \rangle,
\end{equation}
where $\pi \colon G \to \mathcal{B}(E)$ is a representation in $\mathrm{Rep}_p(G)$ and $\xi \in E$, $\eta \in E^*$. The space $B_p(G)$ carries a natural norm given by the infimum of the numbers $\|\xi\|\|\eta\|$ over all representations $\pi$ in $\mathrm{Rep}_p(G)$ and $\xi \in E$, $\eta \in E^*$ such that \eqref{eq:bpcharacterisation} holds.

This definition of $B_p(G)$ is due to Runde \cite{MR2196641}, who showed that $B_p(G)$ is a Banach algebra. However, the conventions used above are slightly different, in order to be better suitable to the purposes of this article.

It is known that for $1 < p \leq q \leq 2$ or $2 \leq q \leq p < \infty$, the algebra $B_q(G)$ embeds contractively into $B_p(G)$ \cite{MR2196641}.\\

For fixed $p$, we have the following norm-decreasing inclusions among the aforementioned spaces:
\[
    A_p(G) \subset B_p(G) \subset M_{\mathrm{cb}}A_p(G) \subset \mathrm{WAP}(G) \subset L^\infty(G).
\]

\section{A predual of $X(G)$}
Let $G$ be a locally compact group, and let $\mathrm{WAP}(G)$ be the algebra of weakly almost periodic functions on $G$. In what follows, let $X(G)$ be a bi-invariant linear subspace of $\mathrm{WAP}(G)$ that is closed under complex conjugation and that contains the constants. Suppose that $X(G)$ has a natural norm $\|\cdot\|_X$, that the Fourier--Stieltjes algebra $B(G)$ embeds contractively into $X(G)$ and that $\|\cdot\|_{\infty} \leq \|\cdot\|_X \leq \|\cdot\|_B$. The assumption that $B(G)$ is contained in $X(G)$ guarantees that $X(G)$ ``contains the unitary \hyphenation{rep-resen-ta-tion}representation theory of $G$'' and that it is sufficiently large for our purposes. In particular, $X(G)$ is $\sigma(L^{\infty}(G),L^1(G))$-dense in $L^{\infty}(G)$ (i.e.~weak$^*$-dense with respect to the weak$^*$-topology on $L^{\infty}(G)$).

For $f \in L^1(G)$, set
\[
        \|f\|_Y:=\sup\left\{ \; \bigg\lvert \int_G f(g)\varphi(g)dg \; \bigg\rvert \; \big\vert \; \varphi \in X(G), \; \|\varphi\|_{X} \leq 1 \, \right\}.
\]
Then $\|\cdot\|_Y$ defines a norm on $L^1(G)$, and we denote by $Y(G)$ the completion of $L^1(G)$ with respect to $\|\cdot\|_Y$. The following result provides a criterion for $X(G)$ to have a predual.
\begin{thm} \label{thm:predual}
Let $X(G)$ be as above. Then the following are equivalent:

\begin{itemize}
\item[(i)] The spaces $X(G)$ and $Y(G)^*$ are isometrically isomorphic. 

\item[(ii)] The unit ball $(X(G))_1$ of $X(G)$ is $\sigma(L^{\infty}(G),L^1(G))$-closed.
\end{itemize}

Moreover, in this case, every bounded linear functional $\alpha \colon Y(G) \to \mathbb{C}$ is of the form
\begin{equation}
        \alpha(f) = \int_G f(g)\varphi(g)dg, \qquad f \in L^1(G),
\end{equation}
for some $\varphi \in X(G)$, and $\|\alpha\|=\|\varphi\|_{X}$.
\end{thm}
\begin{proof}
      (i) $\implies$ (ii): Suppose that $Y(G)^*$ and $X(G)$ are isometrically isomorphic. There is a contractive map $\iota \colon L^1(G)\to Y(G)$.  Its adjoint $\iota^* \colon X(G)\to L^\infty(G)$ is the inclusion map, which is weak$^*$-weak$^*$-continuous. Hence, it maps the unit ball of $X(G)$ to a weak$^*$-compact subset of $L^\infty(G)$.
      
      (ii) $\implies$ (i): First note that for every $\varphi \in X(G)$, the linear functional 
      \[
      \alpha_{\varphi} \colon L^1(G) \to \mathbb{C}, \qquad f \mapsto \int_G f(g)\varphi(g)dg
      \]
      uniquely extends to a bounded linear functional on $Y(G)$, and we have that $\|\alpha_{\varphi}\| \leq \|\varphi\|_{X}$. Let $\Psi:X(G)\to Y(G)^*$ denote the contractive linear map given by $\Psi(\varphi) = \alpha_{\varphi}$. We show that $\Psi$ is, in fact, a surjective isometry. 
      
      Let $\alpha \in Y(G)^*$ with $\|\alpha\|=1$. Since $X(G)$ embeds contractively into $L^\infty(G)$, we observe that for every $f\in L^1(G)$, we have $\|f\|_Y\leq\|f\|_1$, and therefore, $\alpha' := \alpha \big\vert_{L^1(G)}$ is a bounded linear functional on $L^1(G)$. Hence, there exists a $\varphi \in L^{\infty}(G)$ such that 
      \[
        \alpha'(f) = \int_G f(g)\varphi(g)dg, \qquad f \in L^1(G).
      \] 
      We now show that $\varphi\in (X(G))_1$, by means of a version of the bipolar theorem. We consider $(X(G))_1$ as a subset of $L^\infty(G)$ equipped with the $\sigma(L^\infty(G),L^1(G))$-topology. By definition of $\|\cdot\|_{Y}$, we observe that the prepolar ${}^{\circ}(X(G))_1$ of $(X(G))_1$ is given by
      \[
        {}^{\circ}(X(G))_1=\{ f\in L^1(G) \mid \|f\|_{Y}\leq1\}.
      \]
      Since $\|\alpha\|=1$, we have that for every $f\in {}^{\circ}(X(G))_1$,
    \[
  \bigg\lvert \int f(g)\varphi(g)dg \; \bigg\rvert = \lvert\alpha^{\prime}(f)\rvert=\lvert\alpha(f)\rvert\leq1.
    \]
    Therefore, $\varphi$ belongs to the polar of ${}^{\circ}(X(G))_1$. Since $(X(G))_1$ is a convex balanced subset of $L^{\infty}(G)$ that is $\sigma(L^\infty(G),L^1(G))$-closed by assumption, the polar of ${}^{\circ}(X(G))_1$ coincides with $(X(G))_1$ by the bipolar theorem (see \cite[Corollary V.1.9]{MR1070713}), so $\varphi$ belongs to $(X(G))_1$. Moreover, $\Psi(\varphi)=\alpha$ and $\|\Psi(\varphi)\|\leq\|\varphi\|_X \leq 1 = \|\alpha\|$. This implies that $\Psi$ is a surjective isometry, as desired. 
\end{proof}

\begin{rem}
\

\begin{itemize}
\item[(i)] Note that in Theorem \ref{thm:predual}, the existence of a predual does not come for free. For example, $\mathrm{WAP}(G)$ itself does not have a predual in general. Indeed, the algebra $\mathrm{WAP}(G)$ is a $C^*$-algebra. If it would have a predual, then by Sakai's theorem, it would be a von Neumann algebra, which is in general not the case.

\item[(ii)] A subspace $X(G)$ as in Theorem \ref{thm:predual} may have several preduals that are not isometrically isomorphic. An easy example of this is the Fourier--Stieltjes algebra $B(\mathbb{T})$ of the circle group $\mathbb{T}$, which in fact coincides with the completely bounded Fourier multipliers $M_{\mathrm{cb}}A(\mathbb{T})$. It is known that
\[
    B(\mathbb{T}) \cong C^*(\mathbb{T})^* \cong c_0(\mathbb{Z})^* \cong \ell^1(\mathbb{Z}).
\]
The space $\ell^1(\mathbb{Z})$ is known to have many preduals different from $c_0(\mathbb{Z})$.

\item[(iii)] In the case that $X(G)$ is $M_{\mathrm{cb}}A(G)$, the predual from Theorem \ref{thm:predual} coincides with the predual described for this space in \cite[Proposition 1.10]{MR0784292}. Miao generalised this to the $p$-completely bounded multipliers $M_{\mathrm{cb}}A_p(G)$ \cite{MR2040919}, \cite{MR2457409}, \cite{Miao}. Indeed, he proved that the space $Q_{p,\mathrm{cb}}(G)$ defined as the completion of $L^1(G)$ with respect to the norm
\[
\|f\|_{Q_{p,\mathrm{cb}}}=\sup\bigg\{\bigg\lvert \int_G f(g)\varphi(g)dg \; \bigg\rvert\ \vert\ \varphi\in (M_{\mathrm{cb}}A_p(G))_1\bigg\}
\]
is a predual of $M_{\mathrm{cb}}A_p(G)$.

As a matter of fact, the predual of $M_{\mathrm{cb}}A_p(G)$ already occurred earlier in the literature. To see this, note that in \cite[Theorem 8.6]{MR2606882}, Daws establishes an isometric isomorphism
\begin{equation} \label{eq:FS}
HS_p(G) \cong M_{\mathrm{cb}}A_p(G),
\end{equation}
where $HS_p(G)$ is the Banach space of ``$p$-Herz--Schur multipliers'' as defined by Herz in \cite{MR425511}. (In the notation of \cite{MR425511}, however, the space $HS_p(G)$ is denoted by $B_p(G)$.) For the space $HS_p(G)$, Herz had already constructed a predual. Indeed, in \cite{MR425511}, Herz introduces a Banach space $QF_p(G)$ and a contractive map $Q \colon L^1(G)\to QF_p(G)$ with dense range \cite[Proposition 1]{MR425511}. Then, in \cite[Proposition 2]{MR425511}, he shows that $HS_p(G)$ coincides with the range of the adjoint map $Q^*$, and that $Q^*$ induces an isometric isomorphism 
\[
	HS_p(G) \cong QF_p(G)^*.
\]
Combined with \eqref{eq:FS}, this yields the desired predual.

\item[(iv)] For $B_p(G)$, the predual of Theorem \ref{thm:predual} coincides with the predual of this space described in \cite[Theorem 6.6]{MR2196641}.
\end{itemize}
\end{rem}

\section{Property (T$^*_X$)} \label{sec:weakstarcontinuity}
In this section, we study rigidity properties formulated in terms of the weak$^*$-continuity of invariant means on appropriate function spaces of $G$. 

Let $G$ be a locally compact group, and let $X(G)$ be as in Theorem \ref{thm:predual}. In particular, $X(G)$ has a unique invariant mean $m_X$ and the space $Y(G)$ as defined in Theorem \ref{thm:predual} is a predual of $X(G)$.
\begin{dfn} \label{dfn:tstarx}
        A locally compact group $G$ has property (T$^*_X$) if $m_X$ is $\sigma(X(G),Y(G))$-continuous.
\end{dfn}
When it is clear which space $X(G)$ we consider, we often just use the terminology ``weak$^*$-topology'' instead of $\sigma(X(G),Y(G))$-topology. Note that this topology depends on the chosen predual $Y(G)$. Unless explicitly stated otherwise, we always consider the natural predual from Theorem \ref{thm:predual}.

As mentioned in Section \ref{sec:introduction}, for $X(G)=B(G)$, this property corresponds to Kazhdan's property (T) (see \cite{MR0634144}, \cite{MR0773186}), and for $X(G)=M_{\mathrm{cb}}A(G)$, this property corresponds to property (T$^*$) of Haagerup, Knudby and the first-named author (see \cite{MR3552017}).

The following proposition shows that if $X(G)$ is as in Theorem \ref{thm:predual} (in particular we assume that $X(G)$ ``includes the unitary representation theory'' of $G$, in the sense that $B(G)$ embeds contractively into $X(G)$), then property (T$^*_X$) is a strengthening of property (T).
\begin{prp} \label{prp:Tstarcomparison}
    Let $X(G)$ be a subspace of $\mathrm{WAP}(G)$ satisfying the conditions of Theorem \ref{thm:predual}. If $G$ has property (T$^*_X$), then $G$ has property (T).
\end{prp}
\begin{proof}
        The identity map $\mathrm{id} \colon L^1(G) \to L^1(G)$ extends to a linear contraction $Y(G) \to C^*(G)$. Its adjoint is the inclusion map $\iota \colon B(G) \to X(G)$, and the map $\iota$ is weak$^*$-weak$^*$-continuous.
        
        Suppose that the mean $m_X$ on $X(G)$ is $\sigma(X(G),Y(G))$-continuous. Since the mean $m_B$ on $B(G)$ coincides with the composition $m_X \circ \iota$, the mean $m_B$ is $\sigma(B(G),C^*(G))$-continuous.
\end{proof}
\begin{rem}
By the results recalled in Section \ref{sec:preliminaries}, we know that for every $1<p<\infty$, the space $B(G)$ embeds contractively into $B_p(G)$, and $B_p(G)$ embeds contractively into $M_{\mathrm{cb}}A_p(G)$. Hence, the assumptions of Proposition \ref{prp:Tstarcomparison} hold for the spaces $X(G)=B_p(G)$ and $X(G)=M_{\mathrm{cb}}A_p(G)$.
\end{rem}
It is clear that if $X(G)$ is as in Theorem \ref{thm:predual} and $G$ is a compact group, then $G$ has property (T$^*_X$). Indeed, the map $\varphi \mapsto \langle \varphi , 1 \rangle$ defines the weak$^*$-continuous unique invariant mean on $X(G)$.

In case $G$ is a non-compact group, it is in general difficult to show that the invariant mean on $X(G)$ is weak$^*$-continuous. Indeed, establishing (a strengthening) of property (T) for a group $G$ is usually hard. We study specific cases of $X(G)$ in the setting of Lie groups $G$ in the next sections.

\section{The mean on $M_{\mathrm{cb}}A_p(G)$ and property (T$^*_{M_{\mathrm{cb}}A_p}$)}
In this section, we consider the case of $X(G)$ being the space $M_{\mathrm{cb}}A_p(G)$ of $p$-completely bounded multipliers of $A_p(G)$ (with $1 < p < \infty$), and we study property (T$^*_{M_{\mathrm{cb}}A_p}$) and its permanence properties. We write $Q_{p,\mathrm{cb}}(G)$ for the predual of $M_{\mathrm{cb}}A_p(G)$ as described in Theorem \ref{thm:predual}. Some of the arguments in this section are modifications of results from \cite{MR3552017}.

The following generalises \cite[Lemma 5.8]{MR3552017}. The proof follows mutatis mutandis from the proof given there.
\begin{lem} \label{lem:grouphomomorphism}
Let $G$ and $H$ be locally compact groups and $\rho \colon H\to G$ a continuous group homomorphism, and let $1 < p < \infty$. If $u\in C_c(G)$ is a non-negative function with $\|u\|_1 = 1$ and $u^*=u$ (where $u^*(g)=\overline{u(g^{-1})}\Delta(g^{-1})$), then the linear map $T\colon M_{\mathrm{cb}}A_p(G)\to M_{\mathrm{cb}}A_p(H)$ defined by $\varphi \mapsto (u*\varphi)\circ\rho$ is weak$^*$-weak$^*$-continuous.
\end{lem}

\begin{prp} \label{prp:denseimage}
Let $\rho \colon H \rightarrow G$ be a continuous group homomorphism with dense image, and let $1 < p < \infty$. If $H$ has property (T$^*_{M_{\mathrm{cb}}A_p}$), then so has $G$.
\end{prp}
\begin{proof}
Let $u \in C_c(G)$ be a non-negative function with $\|u\|_1$ and $u^*=u$, and let $T\colon M_{\mathrm{cb}}A_p(G)\to M_{\mathrm{cb}}A_p(H)$ be as in Lemma \ref{lem:grouphomomorphism}. In particular, the map $T$ is weak$^*$-weak$^*$-continuous. Let $m$ denote the (unique) invariant mean on $M_{\mathrm{cb}}A_p(H)$. Then $m'(\varphi) = m(T\varphi)$ defines a linear functional on $M_{\mathrm{cb}}A_p(G)$. In the same way as in the proof of \cite[Proposition 5.9]{MR3552017}, we can show that $m'$ is the (unique) invariant mean on $M_{\mathrm{cb}}A_p(G)$. If $m$ is weak$^*$-continuous, then $m'$ is weak$^*$-continuous as well by the weak$^*$-weak$^*$-continuity of $T$.
\end{proof}
\begin{cor} \label{cor:quotients}
If $G$ is a locally compact group with property (T$^*_{M_{\mathrm{cb}}A_p}$) and $N$ is a closed normal subgroup of $G$, then $G/N$ has property (T$^*_{M_{\mathrm{cb}}A_p}$).
\end{cor}
\begin{lem} \label{lem:submultiplicativityQ}
Let $G_1$ and $G_2$ be locally compact groups, let $1 < p < \infty$, and let $f_1 \in Q_{p,\mathrm{cb}}(G_1)$ and $f_2 \in Q_{p,\mathrm{cb}}(G_2)$. The function $f_1 \times f_2 \colon G_1 \times G_2 \to \mathbb{C}$ given by
\[
    f_1 \times f_2 (g_1,g_2) = f_1(g_1)f_2(g_2)
\]
is an element of $Q_{p,\mathrm{cb}}(G_1 \times G_2)$ and
\begin{align}\label{norm-ineq}
    \|f_1 \times f_2\|_{Q_{p,\mathrm{cb}}(G_1 \times G_2)} \leq \|f_1\|_{Q_{p,\mathrm{cb}}(G_1)} \|f_2\|_{Q_{p,\mathrm{cb}}(G_2)}.
\end{align}
\end{lem}
\begin{proof}
   The proof is a modification of the proof of \cite[Lemma 1.4]{MR0996553} up to some straightforward changes. Indeed, let $p^\prime=\frac{p}{p-1}$. The tensor product space 
   \[
   T(G)=\left(L^2(G)\otimes_{\gamma}L^2(G)\right)\otimes_\gamma\left(L^2(G)\otimes_{\lambda}L^2(G)\right)
   \]
  in the proof of \cite[Lemma 1.4]{MR0996553} is replaced by 
   \[
   T_p(G)=\left(L^p(G)\otimes_{\gamma}L^{p^\prime}(G)\right)\otimes_{\gamma}\left(L^{p^\prime}(G)\otimes_{\lambda}L^p(G)\right),
   \]
    where $\otimes_\gamma$ and $\otimes_\lambda$ denote the projective and the injective tensor product, respectively.
    
Herz proved (see \cite[p.~151]{MR425511}) that $Q_{p,cb}(G)$ (denoted by $QF_p(G)$ in \cite{MR425511}) is the image of the tensor product space $T_p(G)$ under the map $\pi$ defined on simple tensors by the formula
\[
\pi(f_1\otimes f_2\otimes f_3\otimes f_4)=(f_1f_3)* (f_2f_4)^*,
\]
where ${}^*$ denotes the involution of $L^1(G)$. Moreover, a function $f$ lies in $ Q_{p,cb}(G)$ if and only if there exists an $s$ in $T_p(G)$ so that $\pi(s)=f$, and 
\[
\|f\|_{Q_{p,\mathrm{cb}}(G)}=\inf\{\|s\|_{T_p} \mid s\in T_p(G),\,\pi(s)=f\}.
\]
Similar to the the proof of \cite[Lemma 1.4]{MR0996553}, the smallest cross norm \linebreak $\|\mathop{\sum}\limits_{i=1}^m h_i\otimes k_i\|_{\lambda}$, where in our setting $h_i\in L^{p^\prime}(G_1)$ and $k_i\in L^p(G_1)$, is the operator norm $\|L\|$ of the linear map $L$ on $L^p(G_1)$ sending $f$ to $\mathop{\sum}\limits_{i=1}^m\langle f,h_i\rangle k_i$. Now recall that if $L \colon L^p(G_1)\to L^p(G_1)$ and $M \colon L^p(G_2)\to L^p(G_2)$ are bounded operators, then 
   \[
   L\otimes M \colon L^p(G_1\times G_2)\to L^p(G_1\times G_2)
   \]
   is a bounded operator with $\| L\otimes M\|=\|L\|\|M\|$. Using this, as in the proof of \cite[Lemma 1.4]{MR0996553}, we obtain that for $s\in T_p(G_1)$ and $t\in T_p(G_2)$,
   \[
   \|s\otimes t\|_{T_p}\leq\|s\|_{T_p}\|t\|_{T_{p}}.
   \]
\end{proof}

The following lemma generalises \cite[Lemma 5.11]{MR3552017} with the same proof.
\begin{lem} \label{lem:meanapproximation}
Let $G$ be a locally compact group with property (T$^*_{M_{\mathrm{cb}}A_p}$), and let $m \in Q_{p,\mathrm{cb}}(G)$ denote the (unique) invariant mean on $M_{\mathrm{cb}}A_p(G)$. Then there exists a sequence $f_n\in L^1(G)$ of non-negative functions with $\|f_n\|_1 = 1$ for all $n \in \mathbb{N}$ such that $\|f_n - m\|_{Q_{p,\mathrm{cb}}(G)} \to 0$.
\end{lem}
\begin{prp}\label{prp:direct product Tstar}
  Let $G_1$ and $G_2$ be two locally compact groups. The direct product $G=G_1 \times G_2$ has property (T$^*_{M_{\mathrm{cb}}A_p}$) if and only if $G_1$ and $G_2$ have property (T$^*_{M_{\mathrm{cb}}A_p}$).
\end{prp}
\begin{proof}
 Suppose that $G_1$ and $G_2$ have property (T$^*_{M_{\mathrm{cb}}A_p}$). For $i=1,2$, let $m_i$ be the invariant mean on $M_{\mathrm{cb}}A_p(G_i)$. For $i=1,2$, by Lemma \ref{lem:meanapproximation}, there are sequences $(f_n^{(i)})$ in $L^1(G_i)_{\geq 0}$ with $\|f_n^{(i)}\|_1=1$ for all $n$ such that $\|f_n^{(i)}-m_i\|_{Q_{p,{\mathrm{cb}}(G_i)}} \to 0$ as $n \to \infty$. For all $g_i \in G_i$, we have
\[
    \|L_{g_i} f_n^{(i)} - f_n^{(i)}\|_{{Q_{p,{\mathrm{cb}}(G_i)}}}\to 0,
\]
because $m_i$ is left invariant. By Lemma \ref{lem:submultiplicativityQ}, the sequence $f_n^{(1)} \times f_n^{(2)}$ is a Cauchy sequence in $Q_{p,\mathrm{cb}}(G)$, and its limit $M$ is a weak$^*$-continuous mean on $M_{\mathrm{cb}}A_p(G)$. It follows that
\[
    \|L_g (f_n^{(1)}\times f_n^{(2)}) - f_n^{(1)}\times f_n^{(2)}\|_{Q_{p,\mathrm{cb}}(G)} \to 0 \qquad\text{for all } g \in G.
\]
Therefore, the mean $M$ is left invariant, so $G$ has property (T$^*_{M_{\mathrm{cb}}A_p}$).

The other direction follows directly from Corollary \ref{cor:quotients}.
\end{proof}

The proof of the following result is a modification of \cite[Proposition 5.13]{MR3552017}.
\begin{prp} \label{prp:modulo-compact}
Let $G$ be a locally compact group, and let $K$ be a compact closed normal subgroup of $G$. Then $G$ has property (T$^*_{M_{\mathrm{cb}}A_p}$) if and only if $G/K$ has property (T$^*_{M_{\mathrm{cb}}A_p}$).
\end{prp}

To conclude this section, we establish a relation between property (T$^*_{M_{\mathrm{cb}}A_p}$) and property (T$^*_{M_{\mathrm{cb}}A_q}$) for different $p$ and $q$.
\begin{prp} \label{prp:tstarpq}
Let $G$ be a locally compact group and $1 < p \leq q \leq 2$ or $2 \leq q \leq p < \infty$. If $G$ has property (T$^*_{M_{\mathrm{cb}}A_p}$), then it has property (T$^*_{M_{\mathrm{cb}}A_q}$).
\end{prp}
\begin{proof}
  Let $1 < p \leq q \leq 2$ or $2 \leq q \leq p < \infty$. Suppose that $G$ has property (T$^*_{M_{\mathrm{cb}}A_p}$). By \cite[Proposition 6.1]{MR2652178}, the inclusion map $$\iota \colon M_{\mathrm{cb}}A_q(G) \to M_{\mathrm{cb}}A_p(G)$$ is a contraction. Its adjoint $\iota^* \colon M_{\mathrm{cb}}A_p(G)^*\to M_{\mathrm{cb}}A_q(G)^*$, also a contraction, maps $Q_{p,\mathrm{cb}}(G)$ to $Q_{q,\mathrm{cb}}(G)$ (see \cite[Proposition 2.1]{MR3706609}). Therefore, $\iota \colon M_{\mathrm{cb}}A_q(G)\to M_{\mathrm{cb}}A_p(G)$ is weak$^*$-weak$^*$-continuous. Let $m$ denote the (unique) invariant mean on $M_{\mathrm{cb}}A_p(G)$. Since $m\big\vert_{M_{\mathrm{cb}}A_q(G)}=m\circ\iota$ is the invariant mean on $M_{\mathrm{cb}}A_q(G)$ and $m$ and $\iota$ are weak$^*$-continuous, we obtain that $G$ has property (T$^*_{M_{\mathrm{cb}}A_q}$).  
\end{proof}

\section{Simple Lie groups with property (T$^*_{M_{\mathrm{cb}}A_p}$)} \label{sec:tstarpsimple}
In this section, we determine exactly which connected simple Lie groups with finite center have property (T$^*_{M_{\mathrm{cb}}A_p}$) for $1 < p < \infty$.

Recall that a (connected) Lie group is called simple if its Lie algebra is simple and that it is called semisimple if its Lie algebra is a direct sum of simple Lie algebras. Let $G$ be a connected semisimple Lie group with finite center, and let $\mathfrak{g}$ denote its Lie algebra. Then $\mathfrak{g}$ has a Cartan decomposition $\mathfrak{g} = \mathfrak{k} + \mathfrak{p}$, where $\mathfrak{k}$ is the Lie algebra of a maximal compact subgroup $K$ of $G$. Furthermore, the Lie group $G$ has a decomposition $G=KAK$, where $A$ is an abelian Lie group, whose Lie algebra $\mathfrak{a}$ is a maximal abelian subspace of $\mathfrak{p}$. The real rank of $G$ is defined as the dimension of $\mathfrak{a}$. For details, see e.g.~\cite{MR1920389} or \cite{MR0746308}. The real rank is an important invariant for our purposes.

Let us first recall two examples of simple Lie groups with real rank $2$ and show that they satisfy property (T$^*_{M_{\mathrm{cb}}A_p}$).

First, let $\mathrm{SL}(3,\mathbb{R})$ denote the special linear group, i.e.~the group of $3 \times 3$-matrices with real entries and determinant $1$. The special orthogonal group $\mathrm{SO}(3)$ is the natural maximal compact subgroup of $\mathrm{SL}(3,\mathbb{R})$.

For the second example, let $J$ be the matrix defined by
\[
  J=\left( \begin{array}{cc} 0 & I_2 \\ -I_2 & 0 \end{array} \right),
\]
where $I_2$ is the identity $2 \times 2$-matrix. The symplectic group $\mathrm{Sp}(2,\mathbb{R})$ is defined as
\[
	\mathrm{Sp}(2,\mathbb{R})=\{g \in \mathrm{GL}(4,\mathbb{R}) \mid g^t J g = J\},
\]
where $g^t$ denotes the transpose of $g$. The group
\[
  K= \left\{ \left( \begin{array}{cc} A & -B \\ B & A \end{array} \right) \in \mathrm{Mat}_{4}(\mathbb{R}) \biggm\vert A+iB \in \mathrm{U}(2) \right\} \cong \mathrm{U}(2)
\]
is a maximal compact subgroup of $\mathrm{Sp}(2,\mathbb{R})$.

Before we can establish property (T$^*_{M_{\mathrm{cb}}A_p}$) for certain specific Lie groups and study it for rather general classes of Lie groups, we recall the following powerful result of Veech \cite[Theorem 1.4]{MR0550072}, which can be seen as a strong version of the Howe--Moore property. First, recall that a Lie group $G$ is semisimple if its Lie algebra $\mathfrak{g}$ is semisimple, i.e.~it decomposes as a direct sum $\mathfrak{g} = \mathfrak{s}_1 \oplus \ldots \oplus \mathfrak{s}_n$, where the algebras $\mathfrak{s}_i$ are simple Lie algebras. We say that a connected semisimple Lie group does not have compact simple factors if for all $i=1,\ldots,n$, the analytic subgroup $S_i$ of $G$ corresponding to the Lie algebra $\mathfrak{s}_i$ is not compact.
\begin{thm}[Veech] \label{thm:veech}
Let $G$ be a connected semisimple Lie group with finite center and without compact simple factors. Then
\[
    \mathrm{WAP}(G)=C_0(G)\oplus \mathbb{C}1
\]
and for every $\varphi \in \mathrm{WAP}(G)$, we have
\[
m(\varphi)=\lim_{g\to\infty}\varphi(g),
\]
where $m$ is the unique invariant mean on $\mathrm{WAP}(G)$.
\end{thm}
\begin{prp} \label{prp:sl3sp2tstarp}
    For $1 < p < \infty$, the groups $\mathrm{SL}(3,\mathbb{R})$ and $\mathrm{Sp}(2,\mathbb{R})$ satisfy property (T$^*_{M_{\mathrm{cb}}A_p}$).
\end{prp}
The proposition essentially follows from \cite[Theorem 1.3 and Theorem 1.4]{MR3706609}. Indeed, in these theorems, Vergara proves that the groups $\mathrm{SL}(3,\mathbb{R})$ and $\mathrm{Sp}(2,\mathbb{R})$ do not satisfy the $p$-AP for $1 < p < \infty$. The $p$-AP was introduced in \cite{MR2652178} as a $p$-analogue of the Approximation Property of Haagerup and Kraus (AP), which goes back to \cite{MR1220905}. It was proved by Lafforgue and de la Salle that $\mathrm{SL}(3,\mathbb{R})$ does not have the AP \cite{MR2838352}. This was generalised to all connected simple Lie groups with real rank at least $2$, including the group $\mathrm{Sp}(2,\mathbb{R})$, by Haagerup and the first-named author in \cite{MR3047470}, \cite{MR3453357}.

We now present a brief proof sketch of Proposition \ref{prp:sl3sp2tstarp}, relying on Vergara's work.
\begin{proof}[Proof sketch of Proposition \ref{prp:sl3sp2tstarp}]
In fact, what is proved in \cite[Theorem 1.3 and Theorem 1.4]{MR3706609} is  stronger than the fact that the groups $\mathrm{SL}(3,\mathbb{R})$ and $\mathrm{Sp}(2,\mathbb{R})$ do not have the $p$-AP for $1 < p < \infty$. We explain this for the group $\mathrm{SL}(3,\mathbb{R})$. The case $\mathrm{Sp}(2,\mathbb{R})$ follows analogously.

Let $G=\mathrm{SL}(3,\mathbb{R})$, and let $K=\mathrm{SO}(3)$ (its maximal compact subgroup). For $g \in G$, let $m_g$ be the measure defined by
\[
    \int_G \varphi\, dm_g = \int_K \int_K \varphi(k_1gk_2)dk_1dk_2.
\]
Vergara shows that $(m_g)$ is a Cauchy net in $M_{\mathrm{cb}}A_p(G)^*$, which converges to a mean $\mu$ on $M_{\mathrm{cb}}A_p(G)$. Additionally, he proves that $\mu$ is actually an element of $Q_{p,\mathrm{cb}}(G)$, i.e.~the mean $\mu$ is weak$^*$-continuous.

We now show that $\mu$ actually coincides with the unique invariant mean $m$ on $M_{\mathrm{cb}}A_p(G)$. From the characterisation of $M_{\mathrm{cb}}A_p(G)$ of Proposition \ref{prp:charpcbfm}, it follows that the function $\Phi \colon g \mapsto \int_K \int_K \varphi(k_1gk_2)dk_1dk_2$ is an element of $M_{\mathrm{cb}}A_p(G)$. Indeed, writing $\varphi(h^{-1}g)=\langle \beta(h),\alpha(g) \rangle$, we can move each integral over $K$ into one of the maps $\alpha$ and $\beta$. Recall that $M_{\mathrm{cb}}A_p(G)$ is contained in $\mathrm{WAP}(G)$ (see Section \ref{sec:preliminaries}). By Theorem \ref{thm:veech}, it now follows that
\begin{equation} \label{eq:veechexample}
    \lim_{g \to \infty} \Phi(g) = m(\Phi)
\end{equation}
Now, let $\varphi \in M_{\mathrm{cb}}A_p(G)$. For all $g \in G$, we have
\[
    |\mu(\varphi)-m(\varphi)| \leq |\mu(\varphi)-m_g(\varphi)| + |m_g(\varphi)-m(\varphi)|.
\]
By Vergara's result, asserting that $(m_g)$ converges to $\mu$ in $M_{\mathrm{cb}}A_p(G)^*$, the first term of the right-hand side tends to $0$ as $g \to \infty$, and by \eqref{eq:veechexample}, the second term tends to $0$. This shows that the weak$^*$-continuous mean $\mu$ coincides with the left-invariant mean $m$, showing that $\mu$ is indeed left-invariant.
\end{proof}

\begin{rem}
Note that the fact that $\mathrm{SL}(3,\mathbb{R})$ and $\mathrm{Sp}(2,\mathbb{R})$ have property (T$^*$) (which is exactly property (T$^*_{M_{\mathrm{cb}}A_p}$) for $p=2$) was proved in \cite{MR3552017}. Alternative to the proof sketch of Proposition \ref{prp:sl3sp2tstarp} given above, one can modify the proofs from \cite{MR3552017} to the $p$-setting. This would require some straightforward technical modifications.
\end{rem}

We are now ready to prove the main result of this section.
\begin{thm} \label{thm:Tstarpsimplegroups}
    Let $G$ be a connected simple Lie group with finite center, and let $1 < p < \infty$. Then $G$ has property (T$^*_{M_{\mathrm{cb}}A_p}$) if and only if the real rank of $G$ is $0$ or at least $2$.
\end{thm}
\begin{proof}
Let $G$ be a connected simple Lie group with finite center. Then $G$ has real rank $0$ if and only if it is compact, in which case it has property (T$^*_{M_{\mathrm{cb}}A_p}$) (see Section \ref{sec:weakstarcontinuity}). If $G$ has real rank $1$, then it is well known that $G$ is weakly amenable \cite{MR0996553}, and hence it has the AP of Haagerup and Kraus, so it cannot have property (T$^*$), since by \cite[Proposition 5.5]{MR3552017}, a locally compact group having the AP and property (T$^*$) has to be compact. By Proposition \ref{prp:tstarpq}, the group $G$ cannot have  property (T$^*_{M_{\mathrm{cb}}A_p}$) for any $p \in (1,\infty)$.

Hence, it only remains to consider the case of real rank at least $2$, which follows in the same way as \cite[Theorem D]{MR3552017}. Let us give a brief argument. Let $G$ be a connected simple Lie group with finite center and real rank at least $2$. It is well known that $G$ contains a closed subgroup $H$ that is locally isomorphic to $\mathrm{SL}(3,\mathbb{R})$ or $\mathrm{Sp}(2,\mathbb{R})$ (i.e.~the Lie algebra of $H$ is isomorphic to the Lie algebra of $\mathrm{SL}(3,\mathbb{R})$ or to the Lie algebra of $\mathrm{Sp}(2,\mathbb{R})$). Let $m_H$ denote the invariant mean on $\mathrm{WAP}(H)$, the restriction of which to $M_{\mathrm{cb}}A_p(H)$ is weak$^*$-continuous by Proposition \ref{prp:sl3sp2tstarp}.

We define a map from $\mathrm{WAP}(G)$ to $\mathrm{WAP}(H)$ in the following way, similar to Lemma \ref{lem:grouphomomorphism}. Let $u \in C_c(G)$ be a non-negative function such that $\|u\|_1=1$ and $u^*=u$. Define the map $T \colon \mathrm{WAP}(G) \to \mathrm{WAP}(H)$ by $T\varphi = (u * \varphi)\big\vert_H$. Since $T(C_0(G)) \subset C_0(H)$, we have $(m_H \circ T) \vert_{C_0(G)} \equiv 0$. Using Theorem \ref{thm:veech}, it follows that the mean $m_G$ on $\mathrm{WAP}(G)$ is given by $m_G=m_H \circ T$. By Lemma \ref{lem:grouphomomorphism}, we see that the restriction of the map $T$ to $M_{\mathrm{cb}}A_p(G)$ is weak$^*$-weak$^*$-continuous, and the result follows.
\end{proof}

\section{Property (T$^*_{M_{\mathrm{cb}}A_p}$) for connected Lie groups}
In this section, we study property (T$^*_{M_{\mathrm{cb}}A_p}$) for connected Lie groups. Let us first recall some structure theory for connected Lie groups, additional to the structure theory for semisimple Lie groups which we recalled before.

Let $G$ be a connected Lie group with Lie algebra $\mathfrak{g}$. If $\mathfrak{r}$ is the solvable radical (i.e.~the largest solvable ideal) of $\mathfrak{g}$, then the Lie algebra $\mathfrak{g}$ can be written as $\mathfrak{g}=\mathfrak{r} \rtimes \mathfrak{s}$, where $\mathfrak{s}$ is a semisimple Lie subalgebra. This decomposition is called the Levi decomposition. Furthermore, as was recalled at the beginning of Section \ref{sec:tstarpsimple}, we can decompose $\mathfrak{s}$ as $\mathfrak{s}=\mathfrak{s}_1 \oplus \cdots \oplus \mathfrak{s}_n$, where the $\mathfrak{s}_i$'s, with $i=1,\ldots,n$, are simple Lie algebras. Let $R$, $S$ and $S_i$, with $i=1,\ldots,n$, be the analytic subgroups of $G$ corresponding to the Lie \hyphenation{al-ge-bras}algebras $\mathfrak{r}$, $\mathfrak{s}$ and $\mathfrak{s}_i$, respectively. Thus, $R$ is a solvable closed normal subgroup of $G$, and $S$ is a (maximal) semisimple analytic subgroup of $G$. However, $S$ is not necessarily a closed subgroup, and in general, we have $G=RS$, but $R \cap S$ can consist of more than one element. However, for simply connected $G$, the subgroup $S$ is, in fact, closed, and the Levi decomposition becomes a semidirect product $G=R \rtimes S$. For details, see \cite[Section 3.18]{MR0746308}.

A covering group of a connected Lie group $G$ is a pair consisting of a Lie group $\widetilde{G}$ and a Lie group homomorphism $\sigma \colon \widetilde{G} \rightarrow G$ that is surjective, in such a way that $(\widetilde{G},\sigma)$ is also a topological covering space. Every connected Lie group $G$ has a universal (i.e.~simply connected) covering group $\widetilde{G}$ (unique up to isomorphism). We refer to \cite[Section I.11]{MR1920389} for more information.

A connected Lie group $G$ is called reductive if its Lie algebra $\mathfrak{g}$ is reductive, i.e.~the direct sum of an abelian Lie algebra and a semisimple Lie algebra. Equivalently, the solvable radical $\mathfrak{r}$ of $\mathfrak{g}$ coincides with the center of $\mathfrak{g}$. If $G$ is reductive, the analytic subgroup $R$ is an abelian Lie group contained in the center $Z(G)$ of $G$. (The center can, however, be larger than $R$.)

Note that for linear algebraic groups over algebraically closed fields, the property of being reductive is defined differently (see e.g.~\cite{MR2458469}). We will not further discuss this.

Before we get to the main theorem, we prove some auxiliary results.
\begin{lem} \label{lem:wstarclosedquotient}
Let $G$ be a locally compact group with a compact normal subgroup $K$. Then $M_{\mathrm{cb}}A_p(G) \cap C_0(G)$ is weak$^*$-closed in $M_{\mathrm{cb}}A_p(G)$ if and only if $M_{\mathrm{cb}}A_p(G/K) \cap C_0(G/K)$ is weak$^*$-closed in $M_{\mathrm{cb}}A_p(G/K)$.
\end{lem}
\begin{proof}
   Suppose that $M_{\mathrm{cb}}A_p(G)\cap C_0(G)$ is weak$^*$-closed in $M_{\mathrm{cb}}A_p(G)$. Let $(\varphi_i)$ be a net in $M_{\mathrm{cb}}A_p(G/K)\cap C_0(G/K)$ that converges to $\varphi \in M_{\mathrm{cb}}A_p(G/K)$ in the weak$^*$-topology on $M_{\mathrm{cb}}A_p(G/K)$. Let $q \colon G\to G/K$ be the canonical quotient map. From the proof of \cite[Proposition 5.1]{MR3706609}, it follows that the operator 
   \[
   T \colon M_{\mathrm{cb}}A_p(G/K)\to M_{\mathrm{cb}}A_p(G), \; \psi \mapsto \psi \circ q
   \]
   is weak$^*$-weak$^*$-continuous. Therefore, $(T(\varphi_i))$ is a net in $M_{\mathrm{cb}}A_p(G)\cap C_0(G)$ that converges in the weak$^*$-topology to $T(\varphi) \in M_{\mathrm{cb}}A_p(G)$. Since $M_{\mathrm{cb}}A_p(G)\cap C_0(G)$ is weak$^*$-closed in $M_{\mathrm{cb}}A_p(G)$, we have $T(\varphi) \in C_0(G)$. Hence, $\varphi$ belongs to $C_0(G/K)$.
   
   For the converse, let $(\varphi_i)$ be a net in $M_{\mathrm{cb}}A_p(G)\cap C_0(G)$ that converges to $\varphi$ in $M_{\mathrm{cb}}A_p(G)$ in the weak$^*$-topology of $M_{\mathrm{cb}}A_p(G)$. Again, it follows from the proof of \cite[Proposition 5.1]{MR3706609} that the operator 
   \[
    \widetilde{T} \colon M_{\mathrm{cb}}A_p(G)\to M_{\mathrm{cb}}A_p(G/K),
   \]
  defined by $\widetilde{T}(\psi)(gK)=\int_{K}\psi(gk) dk$ for $\psi \in M_{\mathrm{cb}}A_p(G)$ and $g\in G$, is weak$^*$-weak$^*$-continuous. Hence, the net $(\widetilde{T}(\varphi_i))$ in $M_{\mathrm{cb}}A_p(G/K)\cap C_0(G/K)$ converges in the weak$^*$-topology to $\widetilde{T}(\varphi)$ in $M_{\mathrm{cb}}A_p(G/K)$. By assumption, $\widetilde{T}(\varphi)$ belongs to $C_0(G/K)$. Hence, $\varphi \in C_0(G)$.
\end{proof}
The following is a generalisation of \cite[Lemma 2.5]{MR1401762}.
\begin{lem} \label{lem:semisimple}
Let $G$ be a connected Lie group with connected closed normal subgroup $N$. Suppose that $M_{\mathrm{cb}}A_p(G) \cap C_0(G)$ is weak$^*$-closed in $M_{\mathrm{cb}}A_p(G)$. Then the center $Z(N)$ of $N$ is compact and $N / Z(N)$ is semisimple.
\end{lem}
\begin{proof}
Since the identity map $L^1(G) \to L^1(G)$ extends to a contraction $Q_{p,\mathrm{cb}}(G) \to C^*(G)$, its Banach adjoint is a weak$^*$-weak$^*$-continuous contraction $B(G) \to M_{\mathrm{cb}}A_p(G)$ that is easily verified to be the inclusion map. Hence, $B(G) \cap C_0(G)$ is weak$^*$-closed in $B(G)$. The claim now directly follows from \cite[Lemma 2.5]{MR1401762}.
\end{proof}
We can now state and prove the main theorem. The proof of this theorem is inspired by \cite[Section 2]{MR1401762}. We also refer to \cite[Section 3.18]{MR0746308} for more details on the structure theory of Lie groups used in the proof.
\begin{thm} \label{thm:mcbapstructureliegroups}
Let $G$ be a connected Lie group, and let $1 < p < \infty$. Suppose that the semisimple part $S$ of the Levi decomposition of $G$ has finite center. Then the following are equivalent:
\begin{enumerate}[(i)]
        \item The group $G$ is a reductive Lie group with compact center satisfying property (T$^*_{M_{\mathrm{cb}}A_p}$).
        \item The space $M_{\mathrm{cb}}A_p(G) \cap C_0(G)$ is closed in $M_{\mathrm{cb}}A_p(G)$ in the\linebreak $\sigma(M_{\mathrm{cb}}A_p(G),Q_{p,\mathrm{cb}}(G))$-topology.
        \item The group $G$ is a reductive Lie group with compact centre, in which every simple factor has real rank $0$ or at least $2$.
\end{enumerate}
\end{thm}
\begin{proof}[Proof of Theorem \ref{thm:mcbapstructureliegroups}]
\
(i) $\implies$ (ii): Suppose that $G$ is a connected reductive Lie group with compact center satisfying property (T$^*_{M_{\mathrm{cb}}A_p}$). Let $\mathfrak{s}$ denote the Lie algebra of $S$, and let $\mathfrak{s}=\mathfrak{s}_1 \oplus \ldots \oplus \mathfrak{s}_n$ be the decomposition of $\mathfrak{s}$ as a direct sum of simple Lie algebras $\mathfrak{s}_i$, with $i=1,\ldots,n$. Let $\widetilde{G}=\mathbb{R}^d\times \widetilde{S}_1 \times \ldots \times \widetilde{S}_n$, where $\widetilde{S}_i$ is the unique simply connected simple Lie group with Lie algebra $\mathfrak{s}_i$ and $\mathbb{R}^d$ is the universal covering group of the (abelian and compact) solvable radical. It is known that this $\widetilde{G}$ is the universal covering group of $G$ (unique up to isomorphism). There exists a discrete subgroup $\Gamma$ of the center $Z(\widetilde{G})=\mathbb{R}^d\times Z(\widetilde{S}_1) \times \ldots \times Z(\widetilde{S}_n)$ of $\widetilde{G}$ such that $G=\widetilde{G}/\Gamma$. Hence, we have
\[
G/Z(G)=\widetilde{G}/Z(\widetilde{G})=\left( \widetilde{S}_1 / Z(\widetilde{S}_1) \right) \times \ldots \times \left( \widetilde{S}_n / Z(\widetilde{S}_n) \right).
\]
It follows that $G/Z(G)$ is semisimple, since the Lie algebra of the right-hand side is $\mathfrak{s}$, which is a semisimple Lie algebra. Also, $G/Z(G)$ is centerless by construction. The quotient $G/Z(G)$ may have compact simple factors, say, without loss of generality, $\widetilde{S}_1 / Z(\widetilde{S}_1),\, \ldots,\, \widetilde{S}_m / Z(\widetilde{S}_m)$ in the right-hand side above. Taking the quotient of $G/Z(G)$ with respect to the compact normal subgroup $\widetilde{S}_1 / Z(\widetilde{S}_1) \times \ldots \times \widetilde{S}_m / Z(\widetilde{S}_m)$, we obtain the quotient group
\[
    H:=(G/Z(G))\,/\,(\widetilde{S}_1/Z(\widetilde{S}_1) \times \ldots \times \widetilde{S}_m / Z(\widetilde{S}_m)),
\]
which is a semisimple Lie group with trivial center and without compact simple factors.

Using Lemma \ref{lem:wstarclosedquotient} twice (i.e.~once for the quotient with respect to $Z(G)$ and once for the quotient with respect to the compact simple factors), we see that it suffices to show that $M_{\mathrm{cb}}A_p(H)\cap C_0(H)$ is weak$^*$-closed in $M_{\mathrm{cb}}A_p(H)$. To this end, let $(\varphi_i)$ be a net in $M_{\mathrm{cb}}A_p(H) \cap C_0(H)$ that converges to $\varphi \in M_{\mathrm{cb}}A_p(H)$ in the weak$^*$-topology. Since $H$ is semisimple, has trivial center and does not have compact simple factors, it now follows from Theorem \ref{thm:veech} that
\[
    m(\varphi)=\lim_{g\to\infty}\varphi(g)\qquad\text{and}\qquad  m(\varphi_i)=\lim_{g\to\infty}\varphi_i(g) \quad \textrm{for all }i,
\]
where $m$ denotes the unique invariant mean on $M_{\mathrm{cb}}A_p(H)$. Since $\varphi_i$ belongs to $C_0(H)$, we have $m(\varphi_i)=0$ for all $i$. Since property (T$^*_{M_{\mathrm{cb}}A_p}$) passes to quotients, the quotient group $H$ also has property (T$^*_{M_{\mathrm{cb}}A_p}$), and hence, we also have $m(\varphi)=\lim_i m(\varphi_i)=0$, so $\varphi$ belongs to $C_0(H)$.

(ii) $\implies$ (i): Suppose that $M_{\mathrm{cb}}A_p(G) \cap C_0(G)$ is weak$^*$-closed in $M_{\mathrm{cb}}A_p(G)$. By Lemma \ref{lem:semisimple}, we know that $Z(G)$ is compact (abelian) and $G/Z(G)$ is semisimple. This implies that the Lie algebra $\mathfrak{g}$ of $G$ is the direct sum of a semisimple Lie algebra (the Lie algebra of $G/Z(G)$) and an abelian Lie algebra, implying that $G$ is a reductive Lie group.

In the same way as in the implication ``(i) $\implies$ (ii)'', by considering the group $G/Z(G)$ and additionally taking the quotient with respect to the compact simple factors, we obtain a semisimple Lie group $H$ with finite center and without compact factors. By applying Lemma \ref{lem:wstarclosedquotient} twice, the space $M_{\mathrm{cb}}A_p(H) \cap C_0(H)$ is weak$^*$-closed in $M_{\mathrm{cb}}A_p(H)$. By Theorem \ref{thm:veech}, the kernel of the unique invariant mean on $M_{\mathrm{cb}}A_p(H)$ coincides with $M_{\mathrm{cb}}A_p(H) \cap C_0(H)$. Hence, this kernel is weak$^*$-closed, so the mean on $M_{\mathrm{cb}}A_p(H)$ is weak$^*$-continuous. This means that the group $H$ has property (T$^*_{M_{\mathrm{cb}}A_p}$). Hence, by applying Proposition \ref{prp:modulo-compact} twice, the group $G$ has property (T$^*_{M_{\mathrm{cb}}A_p}$) as well.

(i) $\iff$ (iii): Again, as in the implication ``(i) $\implies$ (ii)'', we consider the universal covering group $\widetilde{G}=\mathbb{R}^d\times \widetilde{S}_1 \times \ldots \times \widetilde{S}_n$ of $G$, where $\widetilde{S}_1, \ldots, \widetilde{S}_n$ are simply connected simple Lie groups, and let $\Gamma$ be the discrete subgroup of $Z(\widetilde{G})$ with $G=\widetilde{G}/\Gamma$. Hence,
\[
G/Z(G)=\widetilde{G}/Z(\widetilde{G})=\left( \widetilde{S}_1 / Z(\widetilde{S}_1) \right) \times \ldots \times \left( \widetilde{S}_n / Z(\widetilde{S}_n) \right).
\]
Since $Z(G)$ is compact, the group $G$ has property (T$^*_{M_{\mathrm{cb}}A_p}$) if and only if $G/Z(G)$ has property (T$^*_{M_{\mathrm{cb}}A_p}$), by Proposition \ref{prp:modulo-compact}. It follows from Proposition \ref{prp:direct product Tstar} that $G/Z(G)$ has property (T$^*_{M_{\mathrm{cb}}A_p}$) if and only if each $\widetilde{S}_i/Z(\widetilde{S}_i)$ has property (T$^*_{M_{\mathrm{cb}}A_p}$). From Theorem \ref{thm:Tstarpsimplegroups}, we know that $\widetilde{S}_i/Z(\widetilde{S}_i)$ has property (T$^*_{M_{\mathrm{cb}}A_p}$) if and only if $\widetilde{S}_i/Z(\widetilde{S}_i)$, which is centerless simple Lie group (it is simple because its Lie algebra is $\mathfrak{s}_i$, which is a simple Lie algebra), has real rank $0$ or real rank at least $2$.
\end{proof}
\begin{rem}
It is an open question whether there exists a group with property (T$^*$), but without property (T$^*_{M_{\mathrm{cb}}A_p}$) for some $p\neq2$.
\end{rem}

\section{Property (T$^*_{B_p}$) for connected Lie groups}
In this section, we study property (T$^*_{B_p}$) for connected Lie groups, and we prove Theorem \ref{thm:bpstructureliegroups}, which is analogous to (but less explicit than) Theorem \ref{thm:mcbapstructureliegroups}. Indeed, we prove the equivalence of the analogues of the first two equivalent assertions of Theorem \ref{thm:mcbapstructureliegroups}.

First, we establish the following lemma.
\begin{lem} \label{lem:bpw*closedandquotient}
Let $G$ be a locally compact group with a compact normal subgroup $K$. Then $B_p(G) \cap C_0(G)$ is weak$^*$-closed in $B_p(G)$ if and only if $B_p(G/K) \cap C_0(G/K)$ is weak$^*$-closed in $B_p(G/K)$.
\end{lem}
\begin{proof}
By the same arguments as given in the proof of \cite[Proposition 5.1]{MR3706609}, the operator $T \colon B_p(G/K) \to B_p(G), \; \psi \mapsto \psi \circ q$ and the operator $\widetilde{T} \colon B_p(G)\to B_p(G/K)$ defined by $\widetilde{T}(\psi)(gK)=\int_{K}\psi(gk) dk$ are weak$^*$-weak$^*$-continuous.
\end{proof}

We have now established the right structural properties for $B_p(G)$ and permanence properties for property (T$^*_{B_p}$). The proof of the following theorem is similar to the corresponding parts of the proof of Theorem \ref{thm:mcbapstructureliegroups}. We only explain the relevant points of the proof.
\begin{thm} \label{thm:bpstructureliegroups}
Let $G$ be a connected Lie group, and let $1 < p < \infty$. Suppose that the semisimple part $S$ of the Levi decomposition of $G$ has finite center. Then the following are equivalent:
\begin{enumerate}[(i)]
    \item The group $G$ is a reductive Lie group with compact center satisfying property (T$^*_{B_p}$).
    \item The space $B_p(G) \cap C_0(G)$ is closed in $B_p(G)$ in the $\sigma(B_p(G),B_p(G)_{*})$-topology. 
\end{enumerate}
\end{thm}
The proof of this theorem follows mutatis mutandis from the proof of Theorem \ref{thm:mcbapstructureliegroups}.
\begin{rem}
Theorem \ref{thm:bpstructureliegroups} generalises \cite[Theorem 2.7]{MR1401762}. An important difference with the situation of property (T), as covered in \cite{MR1401762}, however, is that it is known exactly which simple Lie groups have property (T). Indeed, let $G$ be a connected simple Lie group. Then $G$ has property (T) if and only if $G$ has real rank $0$, real rank at least $2$ or if $G$ is locally isomorphic to $\mathrm{Sp}(n,1)$ (with $n \geq 2$) or to $F_{4(-20)}$ (in the latter two cases, $G$ has real rank $1$).

It is clear that if $G$ has real rank $0$ or at least $2$, then $G$ has property (T$^*_{B_p}$), which follows from the fact that in these cases, $G$ has property (T$^*_{M_{\mathrm{cb}}A_p}$) (see Theorem \ref{thm:Tstarpsimplegroups}), because property (T$^*_{M_{\mathrm{cb}}A_p}$) is stronger than property (T$^*_{B_p}$).

However, if $G$ is locally isomorphic to $\mathrm{Sp}(n,1)$ (with $n \geq 2$) or to $F_{4(-20)}$, we do not know what happens, although we expect that these groups have property (T$^*_{B_p}$) for $1 < p < \infty$. In the remaining cases of simple Lie groups with real rank $1$, i.e.~groups locally isomorphic to $\mathrm{SO}_0(n,1)$ or $\mathrm{SU}(n,1)$ (with $n \geq 2$), we know that these groups have the Haagerup property, in which case they cannot have property (T$^*_{B_p}$) for $1 < p < \infty$.
\end{rem}

\end{document}